\documentclass[reqno]{amsart}
\usepackage{graphicx,xcolor,cite,mathtools,bbold,bbding}
\usepackage{amssymb,amsmath,amsthm,mathrsfs,paralist,bm,esint,setspace}
\usepackage[ngerman,french,english]{babel}
\RequirePackage[colorlinks,citecolor=blue,urlcolor=blue]{hyperref}
\usepackage{cite}
\usepackage{cases}
\usepackage{tcolorbox}
\usepackage{tikz}
\usepackage{pgfplots}
\usepackage{subcaption}
\usetikzlibrary{arrows.meta}
\pgfplotsset{compat=1.14}
\pgfplotsset{every x tick label/.append style={font=\footnotesize, yshift=0.6ex}}
\pgfplotsset{every y tick label/.append style={font=\footnotesize, xshift=0.5ex}}

\DeclareMathOperator{\dimh}{dim_{_{\rm H}}}

\DeclareMathOperator{\dimmm}{\underline{dim}_{_{\rm M}}\!}

\newcommand{\N}{\mathbb{N}}
\newcommand{\Z}{\mathbb{Z}}

\newcommand{\R}{\mathbb{R}}

\renewcommand{\P}{\mathrm{P}}

\newcommand{\E}{\mathrm{E}}

\newcommand{\1}{\mathbb{1}}
\renewcommand{\d}{{\rm d}}

\newcommand{\e}{{\rm e}}

\renewcommand{\ge}{\geqslant}
\renewcommand{\le}{\leqslant}
\definecolor{CYL}{rgb}{0.3,0.1,0.1}
\definecolor{DK}{rgb}{0.5,0.3,0.5}
\author[D. Khoshnevisan and C.Y. Lee]{Davar Khoshnevisan \and Cheuk Yin Lee}
\address{Department of Mathematics, The University of Utah, Salt Lake City, Utah 84112-0090,
	USA}
\email{davar@math.utah.edu}
\address{School of Science and Engineering, The Chinese University of Hong Kong (Shenzhen), Longgang,
	Shenzhen, Guangdong, 518172, China}
\email{leecheukyin@cuhk.edu.cn}
\title[Slow points of fBm]{On the slow points of fractional Brownian motion}\thanks{
	Research supported in part by the National Science Foundation grant
        DMS-2245242 and the Shenzhen Peacock grant 2025TC0013.}
\newtheorem{stat}{Statement}[section]
\newtheorem{proposition}[stat]{Proposition}

\newtheorem{corollary}[stat]{Corollary}
\newtheorem{theorem}[stat]{Theorem}
\newtheorem{lemma}[stat]{Lemma}
\theoremstyle{definition}

\numberwithin{equation}{section}
\usepackage[cal=dutchcal,calscaled=1.05,
	scr=boondoxo,scrscaled=1.05,
	]{%
mathalfa}
\date{June 26, 2026}
\begin{document}
\maketitle
\setcounter{tocdepth}{3}
\let\oldtocsection=\tocsection
\let\oldtocsubsection=\tocsubsection
\let\oldtocsubsubsection=\tocsubsubsection

\begin{abstract}
	Esser and Loosveldt \cite{EsserLoosveldt} have recently resolved a long-standing open problem
	in the folklore by proving that fractional Brownian motion (fBm) has slow points in the sense of
	Kahane \cite{Kahane76}, following a rich theory of slow points developed for 
	Brownian motion and other, related, self-similar Markov processes \cite{Perkins,
	Davis,Dvoretzky,Kahane76,Kahane83,Marsalle,
	BassBurdzy,GreenwoodPerkinsB,GreenwoodPerkinsA,Verzani}. We presently introduce
	another method for the study of slow points in order to compute the Hausdorff dimension
	of fBm slow points. Our method follows recent ideas on the points of slow
	growth for SPDEs \cite{KL2} but also requires a number of new
	localization ideas that are likely to have other applications.
\end{abstract}

\section{Introduction}

Let us choose and fix a number $H\in(0\,,1)$, and recall \cite{Kolmogorov,Schoenberg,VS,MandelbrotVanNess}
that a stochastic process $B=\{B(t)\}_{t\ge0}$ is a \emph{fractional Brownian motion {\rm (fBm)} with index
$H$} if $B$ is a centered Gaussian process such that $B(0)=0$ and 
$\E(|B(t)-B(s)|^2)=|t-s|^{2H}$ for every $s,t\ge0$ . 
It is well known that, up to a modification, the random function $t\mapsto B(t)$ is
H\"older continuous of index $\gamma$ for every $\gamma\in(0\,,H)$; see  \cite{MandelbrotVanNess}.
This is sharp since
$\P\{\limsup_{h\to0^+}|B(t+h)-B(t)|/h^H=\infty\}=1$ for every $t>0$ thanks to scaling.

A point $t>0$ is said to
be a \emph{slow point} -- see Kahane \cite{Kahane76} -- if 
\[\textstyle
	0<\limsup_{h\to0^+} h^{-H} |B(t+h)-B(t)|<\infty\quad\text{with probability one}.
\]
Esser and Loosveldt \cite{EsserLoosveldt} have recently developed delicate wavelet techniques
which imply that $\inf_{t>0}\limsup_{h\to0^+}h^{-H}  |B(t+h)-B(t)| <\infty$ a.s.
This implies the existence of slow points because one can apply
a well-known argument  of Dvoretzky \cite{Dvoretzky} in order to see that
$\inf_{t>0}\limsup_{h\to0^+} h^{-H}|B(t+h)-B(t)| >0$ a.s.

The case $H=\frac12$ (Brownian motion) was studied earlier and has a particularly distinguished
history \cite{Perkins,Davis,Dvoretzky,Kahane76,Kahane83,BassBurdzy,GreenwoodPerkinsB,GreenwoodPerkinsA}.
In this case, the process is strong Markov, a fact that plays a key role
in most of the earlier studies of slow points. Our goal is to introduce a new technique for a 
more in-depth analysis of the slow points of self-similar
Gaussian processes. We specifically develop that method
here for fBm in order to simplify the exposition and leave other, potentially
interesting, generalizations to the interested reader.
Our effort also addresses a concern
of Esser and Loosveldt \cite{EsserLoosveldt} about the lack of other available 
methods for deeper analysis of fBm slow points by supplying another such method.

Davis \cite{Davis} and  Perkins \cite{Perkins} have
observed that a key first step in the study of slow points is
the development of boundary-crossing probabilities, 
found earlier independently by Breiman \cite{Breiman} and
Shepp \cite{Shepp}. Recently, we developed general boundary-crossing
probability estimates \cite{KL1} that imply the following for fBm:
There exists a strictly decreasing, continuous, 
function $\lambda: (0\,,\infty) \to (0\,,\infty)$ 
such that: $\lim_{\theta \to 0^+}\lambda(\theta) = \infty$;
$\lim_{\theta \to \infty} \lambda(\theta) = 0$; and
for every $\theta > 0$, 
\begin{equation}\label{BC}\textstyle
	\P\left\{ |B(s)| \le \theta s^{H}\quad \forall s\in[h\,,1]\right\} = 
	h^{\lambda(\theta)+\mathscr{o}(1)}\qquad\text{as $h\to0^+$}.
\end{equation}
The precise values of $\lambda$ are not known, except that it is known that 
$\lambda^{-1}(\theta)$ is the smallest positive root of a modified Kummer function
when $H=\frac12$ (Brownian motion); see Davis \cite{Davis} and Shepp \cite{Shepp}.
Moreover,
$\log\lambda(c)\asymp -c^2$ and
$\lambda(b)\asymp b^{-1/H}$ uniformly for all $c>1$ and $b\in(0\,,1)$; see 
Corollary 1.2 of Ref.~\cite{KL1}.

The purpose of this paper is to show how one can use \eqref{BC} in order
to study the slow points of fBm.
An analogous undertaking was taken up recently in \cite{KL2} in order to exhibit the existence of
a family of ``points of slow growth'' for stochastic PDEs.
Parts of the methods of the present paper draw on ideas in \cite{KL2}. Therefore, the bulk
of this effort is concerned with highlighting the portions of the arguments that are genuinely new.

The following is the main result of this paper.

\begin{theorem}\label{th:main}
	Choose and fix a compact set $K \subset (0\,,\infty)$.
	Then, almost surely,
	\begin{equation}\label{main}
		\lambda^{-1}( \dimmm K) \le  \adjustlimits\inf_{t \in K} 
		\limsup_{h \to 0^+} \frac{|B(t+h)-B(t)|}{h^H}
		\le \lambda^{-1}(\dimh K),
	\end{equation}
	where $\dimmm$ and $\dimh$ respectively denote 
	the lower Minkowski dimension and the Hausdorff dimension 
	{\rm\cite{Falconer}}.
\end{theorem}
For every $\theta>0$ let $\mathfrak{S}(\theta)$ denote the collection of all \emph{$\theta$-slow points}
of fractional Brownian motion. That is, let us consider the random set
$\mathfrak{S}(\theta)$, defined as the collection of all times $t>0$ such that
$\limsup_{h\to0^+} h^{-H}  |B(t+h)-B(t)| \le \theta$.
Thanks to scaling and Fubini's theorem, 
$\mathfrak{S}(\theta)$ has a.s.~zero Lebesgue measure.
The following is a ready consequence of Theorem \ref{th:main} and is presented
here without proof
since one can prove the next result using the same ideas that appear
already in \cite{KL2}.

\begin{corollary}
	$\dimh \mathfrak{S}(\theta)=1-\lambda(\theta)$ a.s.~for every
	non-random  $\theta>0$, where
	$\dimh A<0$ means that $A=\varnothing$.
\end{corollary}

Throughout, $\|X\|_k$ denotes the $L^k(\P)$-norm of a random variable $X$
for every real number $k \ge 1$,
and $\log_+(x) := \log(x \vee \e)$ for every $x > 0$.

\section{Localization}
A central portion of the proof of Theorem \ref{th:main} has to do with the
fact that the increments of fBm are ``approximately independent,''
suitably interpreted. The results of this section set the stage for clarifying
what exactly ``approximate independence'' means in the present context.
As such, we believe that the following might have  other potential uses as well.

Recall from Mandelbrot and Van Ness \cite{MandelbrotVanNess} the following representation for $B$:
\begin{align}\label{fbm:rep}
	B(t) = K_H \left[ \int_0^\infty \left[ (t+s)^{H-\frac12} 
	- s^{H-\frac12} \right] \d \xi(s) + \int_0^t (t-s)^{H-\frac12}\, \d \tilde\xi(s) \right],
\end{align}
where $\xi$ and $\tilde\xi$ denote independent white noises on $\R_+$ and $K_H$ is a normalizing constant.
Consequently, for all $t>0$ and $h > 0$,
\begin{align*}
	&B(t+h) - B(t) = K_H \int_0^\infty \left[ (t+h+s)^{H-\frac12} 
	- (t+s)^{H-\frac12} \right] \d \xi(s)\\
	& + K_H \int_0^t \left[ (t+h-s)^{H-\frac12} - (t-s)^{H-\frac12} \right] \d \tilde\xi(s) 
		+ K_H \int_t^{t+h} (t+h-s)^{H-\frac12}\, \d \tilde\xi(s).
\end{align*}
Let us define for every $\alpha \in (0\,,1)$, $t>0$, and $h \in (0\,,t^{1/\alpha}]$, 
the following localized version of the above increment:
\begin{equation}\label{D_alpha}\begin{aligned}
	D_\alpha(t\,,h) &= K_H \int_{t-h^\alpha}^t \left[ (t+h-s)^{H-\frac12} - (t-s)^{H-\frac12} \right]
		\d \tilde\xi(s)\\
	& \hskip2in + K_H \int_t^{t+h} (t+h-s)^{H-\frac12}\, \d \tilde\xi(s);
\end{aligned}\end{equation}
and consider  the localization error:
\begin{align}\label{error}
	\mathcal{E}_\alpha(t\,,h) = B(t+h)-B(t) - D_\alpha(t\,,h)
	\qquad\forall t>0,\, h \in (0\,, t^{1/\alpha}].
\end{align}
Note that the localized increment $D_\alpha(t\,,h)$ depends only on 
$\{\tilde\xi(s)\}_{s \in [t-h^\alpha, t+h]}$.
Also, note that when $H=\frac12$, we have
$D_\alpha(t\,,h) = B(t+h) - B(t)$ and $\mathcal{E}_\alpha \equiv 0$.

\begin{lemma}\label{lem:D-D:t}
	For every $\alpha \in (0\,,1)$ and $t_1>t_0>0$, 
	\[
		\left\| D_\alpha(t\,,h)-D_\alpha(t',h) \right\|_2 
		\lesssim \begin{cases} 
		|t-t'|^H & \text{if $H \in (0\,,\frac12]$,}\\
		(|t-t'|+h^\alpha)^{H-\frac12} |t-t'|^{\frac12} & \text{if $H \in (\frac12\,,1)$,}
		\end{cases}
	\]
	uniformly for all $t, t' \in [t_0\,,t_1]$ and $h \in (0\,,t_0^{1/\alpha}]$.
\end{lemma}

\begin{proof}
	Without loss of generality, we may and will assume that $t_0 \le t < t' \le t_1$. 
	We now consider the lemma in two different cases.\\
	
	Case 1: $t'-h^\alpha \le t+h$. In this case, we may write
	$D_\alpha(t',h) - D_\alpha(t\,,h) = I_1 - I_2 - I_3$,
	where
	\begin{align*}
		&I_1 = K_H \int_{t'-h^\alpha}^{t'+h} (t'+h-s)^{H-\frac12}\,
			\d \tilde\xi(s) - K_H \int_{t'-h^\alpha}^{t+h} (t+h-s)^{H-\frac12}\, \d \tilde\xi(s),\\
		&I_2 = K_H \int_{t-h^\alpha}^{t'-h^\alpha} (t+h-s)^{H-\frac12}\, \d \tilde\xi(s),\\
		&I_3 = K_H \int_{t'-h^\alpha}^{t'} (t'-s)^{H-\frac12}\, \d \tilde\xi(s) 
			- K_H \int_{t-h^\alpha}^{t} (t-s)^{H-\frac12}\, \d \tilde\xi(s).
	\end{align*}
	It follows from the representation \eqref{fbm:rep} that
	\begin{align*}
		\|I_1\|_2^2 \le \|B(t'+h)-B(t+h)\|_2^2 = |t'-t|^{2H}.
	\end{align*}
	Next, we may observe that
	\begin{align*}
		\|I_2\|_2^2 &= K_H^2 \int_{t-h^\alpha}^{t'-h^\alpha} (t+h-s)^{2H-1} \,\d s
		\lesssim (h+h^\alpha)^{2H} - (h+h^\alpha - (t'-t))^{2H}\\
		&\lesssim \begin{cases}
		(t'-t)^{2H}, & \text{if $H \in (0\,,\frac12]$,}\\
		(h+h^\alpha)^{2H-1} (t'-t) \lesssim h^{\alpha(2H-1)}(t'-t), & \text{if $H \in (\frac12\,,1)$,}
		\end{cases}
	\end{align*}
	uniformly for all $t < t'$ in $[t_0\,,t_1]$ and $h \in (0\,,t_0^{1/\alpha}]$.
	We have used the elementary inequality, $x^p- y^p \le (x-y)^p$, valid for all $x > y>0$ 
	and $p \in (0\,,1)$ -- for the case that $H\in (0\,,\frac12]$ -- and 
	the mean value theorem for the case that $H \in (\frac12\,,1)$.

	As regards $I_3$, we offer
	\begin{align*}
		\|I_3\|_2^2 &\lesssim 
		\begin{cases}
			\int_{t'-h^\alpha}^{t'} (t'-s)^{2H-1}\, \d s+ 
				\int_{t-h^\alpha}^{t} (t-s)^{2H-1}\, \d s & \text{if $t \le t'-h^\alpha$;}\medskip\\
			\int_{t}^{t'} (t'-s)^{2H-1} \,\d s+ \int_{t-h^\alpha}^{t'-h^\alpha} (t-s)^{2H-1} \,\d s & \\
				\qquad
				+ \int_{t'-h^\alpha}^t [(t'-s)^{H-\frac12} - (t-s)^{H-\frac12}]^2 \, \d s 
				& \text{if $t > t'-h^\alpha$}
		\end{cases}\\[5pt]
		&\lesssim 
		\begin{cases}
			h^{2H\alpha} & \text{if $t \le t'-h^\alpha$;}\\
			(t'-t)^{2H} + (h^\alpha + t'-t)^{2H} - (h^\alpha)^{2H} + 
				\|B(t')-B(t)\|_2^2& \text{if $t > t'-h^\alpha$}
		\end{cases}\\[5pt]
		& \lesssim \begin{cases}
			(t'-t)^{2H} & \text{if $h^\alpha \le t'-t$;}\\
			(t'-t)^{2H} + h^{\alpha(2H-1)}(t'-t) & \text{if $t'-t < h^\alpha$,}
		\end{cases}
	\end{align*}
	uniformly for all $t < t'$ in $[t_0\,,t_1]$ and $h \in (0\,,t_0^{1/\alpha}]$,
	where we have used the mean value theorem to obtain the last inequality for the second case.
	It follows that
	\begin{align}\label{D-D:t:I_3}
		\|I_3\|_2^2 \lesssim 
		\begin{cases}
		|t'-t|^{2H} & \text{if $H\in (0\,,\frac12]$,}\\
		(|t'-t| + h^\alpha)^{2H-1} |t'-t| & \text{if $H \in (\frac12\,,1)$,}
		\end{cases}
	\end{align}
	uniformly for all $t < t'$ in $[t_0\,,t_1]$ and $h \in (0\,,t_0^{1/\alpha}]$.
	We combine the estimates for $I_1$, $I_2$, and $I_3$ to deduce the desired result
	in Case 1.\\
	
	Case 2: $t'-h^\alpha > t+h$. In this case,
	$D_\alpha(t'\,,h) - D_\alpha(t\,,h) = I_1 - I_2 - I_3$,
	where
	\begin{align*}
		&I_1 = K_H \int_{t'-h^\alpha}^{t'+h} (t'+h-s)^{H-\frac12}\, \d \tilde\xi(s),\qquad
			I_2 = K_H \int_{t-h^\alpha}^{t+h} (t+h-s)^{H-\frac12}\, \d \tilde\xi(s),\\
		&I_3 = K_H \int_{t'-h^\alpha}^{t'} (t'-s)^{H-\frac12}\, \d \tilde\xi(s) - 
			K_H \int_{t-h^\alpha}^{t} (t-s)^{H-\frac12}\, \d \tilde\xi(s).
	\end{align*}
	In particular, 
	$\|I_1\|_2^2 + \|I_2\|_2^2 \lesssim (h+h^\alpha)^{2H} \le (t'-t)^{2H},$
	and \eqref{D-D:t:I_3} holds just as it did in Case 1, all valid uniformly for
	all $t < t'$ in $[t_0\,,t_1]$ and $h \in (0\,,t_0^{1/\alpha}]$.
	This completes the proof.
\end{proof}

\begin{lemma}\label{lem:D-D:eps}
	For every $\alpha \in (0\,,1)$ and $t_1>t_0>0$, 
	\(
		\|D_\alpha(t\,,h) - D_\alpha(t\,,h')\|_2 \lesssim |h-h'|^{\alpha(H \wedge \frac12)},
	\)
	uniformly for all $t \in [t_0\,,t_1]$ and $h,h' \in (0\,,t_0^{1/\alpha}]$.
\end{lemma}

\begin{proof}
	Without loss of generality, we may and will assume that $0<h<h' \le t_0^{1/\alpha}$. 
	Then, we write $D_\alpha(t\,,h') - D_\alpha(t\,,h) = I_1 + I_2 - I_3$, where
	\begin{align*}
		&I_1 = K_H \int_{t-h^\alpha}^{t+h'} (t+h'-s)^{H-\frac12}\, \d \tilde\xi(s) 
			- K_H \int_{t-h^\alpha}^{t+h} (t+h-s)^{H-\frac12}\, \d \tilde\xi(s),\\
		&I_2 = K_H \int_{t-h'^\alpha}^{t-h^\alpha} (t+h'-s)^{H-\frac12}\, \d \tilde\xi(s),
			\qquad
			I_3 = K_H \int_{t-h'^\alpha}^{t-h^\alpha} (t-s)^{H-\frac12}\, \d \tilde\xi(s).
	\end{align*}
	From \eqref{fbm:rep}, we see that
	$\|I_1\|_2^2 \le \|B(t+h')-B(t+h)\|_2^2 = |h'-h|^{2H}$
	uniformly for all $t \in [t_0\,,t_1]$ and $h \in (0\,,t_0^{1/\alpha}]$.
	Next, we observe that
	\begin{align*}
		\|I_2\|_2^2 &= K_H^2 \int_{t-h'^\alpha}^{t-h^\alpha} (t+h'-s)^{2H-1} \,\d s
		\lesssim (h' + h'^\alpha)^{2H} - (h' + h^\alpha)^{2H}\\
		& \lesssim \begin{cases}
			(h'^\alpha - h^\alpha)^{2H} \le |h'-h|^{2H\alpha}
				& \text{if $H \in (0\,,\frac12]$,}\\
			|h'^\alpha-h^\alpha| (h'+h'^\alpha)^{2H-1} \lesssim |h'-h|^\alpha
				& \text{if $H \in (\frac12\,,1)$,}
		\end{cases}
	\end{align*}
	uniformly for all $t \in [t_0\,,t_1]$ and $h \in (0\,,t_0^{1/\alpha}]$.
	In the case that $H \in (\frac12\,,1)$,
	the last inequality hinges on the mean value theorem;
	and, in both cases, the last inequality hinges also on the fact that 
	$x^p- y^p \le (x-y)^p$ for all $x > y>0$ 
	and $p \in (0\,,1)$. Likewise,
	\[
		\|I_3\|_2^2 
		= K_H^2 \int_{t-h'^\alpha}^{t-h^\alpha} (t-s)^{2H-1} \,\d s \lesssim h'^{2H\alpha} - h^{2H\alpha}
		\lesssim \begin{cases}
			|h'-h|^{2H\alpha}& \text{if $H \in (0\,,\frac12]$,}\\
			|h'-h|^\alpha& \text{if $H \in (\frac12 \,,1)$,}
		\end{cases}
	\]
	uniformly for all $t \in [t_0\,,t_1]$ and $h \in (0\,,t_0^{1/\alpha}]$.
	Combine to conclude the proof.
\end{proof}

\begin{lemma}\label{lem:E}
	Choose and fix $t_0 > 0$ and $\alpha \in (0\,,1)$,
	and set $\beta = \beta(H\,,\alpha) := H+(1-H)(1-\alpha)$. Then,
	$\|\mathcal{E}_\alpha(t\,,h)\|_k 
	\lesssim \sqrt{k}\,h^{\beta}$ uniformly for all $k\in[2\,,\infty)$, 
	$t\ge t_0$, and $h \in (0\,, t_0^{1/\alpha}]$.
	Consequently, there exists a constant $c>0$ such that 
	$\sup_{t \ge t_0}\sup_{h \in (0,t_0^{1/\alpha}]}
	\E\exp(ch^{-2\beta}|\mathcal{E}_\alpha(t\,,h)|^2) < \infty$.
\end{lemma}

\begin{proof}
	For every $t \ge t_0$ and $h\in (0\,,t_0^{1/\alpha}]$, 
	\begin{align*}
		\mathcal{E}_\alpha(t\,,h)  &=K_H \int_0^\infty \left[(t+h+s)^{H-\frac12} 
			- (t+s)^{H-\frac12}\right] \d \xi(s)\\
		& \quad + K_H \int_0^{t-h^\alpha} \left[ (t+h-s)^{H-\frac12} 
			- (t-s)^{H-\frac12} \right] \d \tilde\xi(s) =: I_1 + I_2.
	\end{align*}
	A change of variables and the fundamental theorem of calculus together yield
	\begin{align*}
		&\|I_1\|_2^2 = K_H^2 \int_0^\infty \left[ (t+h+s)^{H-\frac12} 
			- (t+s)^{H-\frac12} \right]^2 \d s\\
		& \propto \int_t^\infty \left[ (h+s)^{H-\frac12} 
			- s^{H-\frac12} \right]^2 \d s = 
			h^{2H} \int_{t/h}^\infty \left[ (1+s)^{H-\frac12} 
			- s^{H-\frac12} \right]^2 \d s\\
		& \propto h^{2H} \int_{t/h}^\infty \left| \int_0^1 (r+s)^{H-3/2}\, \d r \right|^2 \d s 
			\le h^{2H} \int_{t/h}^\infty s^{2H-3}\, \d s
			\propto h^{2H}(t/h)^{2H-2}\lesssim h^2,
	\end{align*}
	uniformly for all $t \ge t_0$ and $h\in (0\,,t_0^{1/\alpha}]$.
	Similarly,
	\begin{align*}
		\|I_2\|_2^2 &= K_H^2 \int_0^{t-h^\alpha} \left[ 
			(t+h-s)^{H-\frac12} - (t-s)^{H-\frac12} \right]^2 \d s\\
		&\le K_H^2 \int_{h^\alpha}^\infty \left[ (h+s)^{H-\frac12} 
			- s^{H-\frac12} \right]^2 \d s
			\lesssim h^{2H} (h^{\alpha-1})^{2H-2} = h^{2\beta},
	\end{align*}
	uniformly for all $t \ge t_0$ and $h \in (0\,,t_0^{1/\alpha}]$. 
	Combine the above estimates and use the fact that
	$\beta < 1$ in order to obtain 
	$\|\mathcal{E}_\alpha(t\,,h)\|_2 \lesssim h^\beta$ uniformly 
	for all $t \ge t_0$ and $h \in (0\,,t_0^{1/\alpha}]$.
	Since $\mathcal{E}_\alpha(t\,,h)$ is Gaussian, the first assertion follows.
	This completes the proof since the second assertion follows from the first assertion 
	together with the Taylor expansion of the exponential function 
	and Stirling's formula.
\end{proof}

\begin{lemma}\label{lem:E-E}
	Choose and fix $t_1 > t_0 > 0$ and $\alpha \in (0\,,1)$, and define
	\[
		\rho((t\,,h)\,,(t',h')) =
		|t-t'|^{H\wedge \frac12} + |h-h'|^{\alpha (H \wedge \frac12)}.
	\]
	Then, $\|\mathcal{E}_\alpha(t\,,h) - \mathcal{E}_\alpha(t',h')\|_2 \lesssim \rho((t\,,h)\,,(t',h'))$
	uniformly for all $t,t' \in [t_0\,,t_1]$ and $h,h' \in (0\,,t_0^{1/\alpha}]$.
	In particular, there exists a constant $c>0$ such that
	\begin{align*}
		\E \exp\left( \sup_{\substack{(t,h),(t',h') \in [t_0,t_1]
		\times (0, t_0^{1/\alpha}]:\\ (t,h)\ne (t',h')}}
		\frac{c\left|\mathcal{E}_\alpha(t\,,h)-\mathcal{E}_\alpha(t',h') \right|^2}{%
		\phi\left( \rho((t\,,h)\,,(t',h'))\right)}
		\right) < \infty,
	\end{align*}
	where $\phi(r) = r^2\log_+(1/r)$ for every $r>0$.
\end{lemma}

\begin{proof}
	It follows readily from the triangle inequality and Lemmas \ref{lem:D-D:t} and \ref{lem:D-D:eps}
	that $\|\mathcal{E}_\alpha(t\,,h) - \mathcal{E}_\alpha(t',h')\|_2
	\lesssim |t-t'|^{H \wedge \frac12} + |h-h'|^{\alpha(H\wedge \frac12)}$
	uniformly for all $t, t' \in [t_0\,,t_1]$ and $h,h' \in (0\,,t_0^{1/\alpha}]$.
	This proves the first assertion.
	The second assertion follows from the first and a standard metric entropy argument.
\end{proof}

\begin{lemma}\label{lem:supE}
	Choose and fix $\alpha \in (0\,,1)$ and $t_1 > t_0>0$,
	and recall $\beta=\beta(\alpha\,,H)$ from Lemma \ref{lem:E}.
	Then,
	\begin{align*}
		\left\| \textstyle\sup_{t \in [t_0,t_1]} 
		\sup_{h \in (0,b]} h^{-H} |\mathcal{E}_\alpha(t\,,h)| \right\|_k 
		\lesssim \sqrt{k}\, b^{\beta-H} |\log b|^{1/2},
	\end{align*}
	uniformly for all $k \in [1\,,\infty)$ and $b \in (0\,,b_0]$, where 
	$b_0 = t_0^{1/\alpha}\wedge \e^{-1}$.
\end{lemma}

\begin{proof}
	Choose and fix $\alpha \in (0\,,1)$ and $t_1>t_0>0$.
	Set $\gamma = H \wedge \frac12$ and, for every $n\in\N$,
	define $F_n$ to be the collection of
	all $2$-D points of the form $ ( in^{-1/\gamma} , jn^{-1/(\alpha\gamma)})$,
	as $i$ and $j$ each range over $\Z_+$.
	Then, 
	\[\textstyle
		\P\left\{ \sup_{t\in [t_0,t_1]} \sup_{h \in (0,b]} 
		|\mathcal{E}_\alpha(t\,,h)| \ge z \right\} \le P_1 + P_2
		\qquad\forall z>0,
	\]
	where 
	\begin{align*}
		P_1 &\textstyle = \P\left\{ \max_{(t,h)\in F_n \cap [t_0,t_1]\times (0,b]} 
			|\mathcal{E}_\alpha(t\,,h)| \ge z/2 \right\},\\
		P_2 &\textstyle= \P\left\{ \sup_{\substack{t, t' \in [t_0,t_1]:\\ |t-t'| \le n^{-1/\gamma}}}
			\sup_{\substack{h,h' \in (0,b]:\\|h-h'| \le n^{-1/(\alpha\gamma)}}} 
			|\mathcal{E}_\alpha(t\,,h) - \mathcal{E}_\alpha(t',h')| \ge z/2 \right\}.
	\end{align*}
	Then, Boole's inequality and Lemma \ref{lem:E} together imply that there exists $c_1>0$ such that
	\begin{align*}
		P_1 &\textstyle \le |F_n \cap [t_0\,,t_1] \times (0\,,b]| 
			\cdot \sup_{t \in [t_0,t_1]} \sup_{h \in (0,b]} 
			\P\left\{ |\mathcal{E}_\alpha(t\,,h)| \ge z/2\right\}\\
		&\le c_1^{-1} n^{\frac{1}{\gamma}+\frac{1}{\alpha\gamma}} \e^{ -c_1 z^2/b^{2\beta}},
	\end{align*}
	uniformly for all $z > 0$, $b \in (0\,,b_0]$ and $n \in \N$, where $|\cdots|$ denotes cardinality.
	By Lemma \ref{lem:E-E}, there exists $c_2>0$ such that
	\begin{align*}
		P_2 &\le \P\left\{ \sup_{\substack{t, t' \in [t_0,t_1]:\\ 
			|t-t'| \le n^{-\frac{1}{\gamma}}}} \sup_{\substack{
			h,h' \in (0,b]:\\|h-h'| \le n^{-\frac{1}{\alpha\gamma}}}} 
			\frac{|\mathcal{E}_\alpha(t\,,h) - \mathcal{E}_\alpha(t',h')|}{
			\rho((t\,,h),(t',h'))\sqrt{\log_+\frac{1}{\rho((t\,,h),(t',h'))}}}
			\ge \frac{z/2}{\frac{2}{n} \sqrt{\log_+ \frac{n}{2}}} \right\}\\
		& \le c_2^{-1} \exp\left( - \frac{c_2 z^2 n^2}{\log_+ n} \right),
	\end{align*}
	uniformly for all $z>0$, $b \in (0\,,b_0]$ and $n \in \N$.
	Let us write
	$L_b = b^\beta \sqrt{|\log b|}$ for every $b \in (0\,,b_0]$.
	It follows that
	\begin{align*}
		&\E\left[ \sup_{t\in [t_0,t_1]} \sup_{h \in (0,b]} 
			|\mathcal{E}_\alpha(t\,,h)|^k \right]
			= k \int_0^\infty z^{k-1} \P\left\{ \sup_{t\in [t_0,t_1]}
			\sup_{h \in (0,b]} |\mathcal{E}_\alpha(t\,,h)| \ge z \right\} \d z\\
		& \le (AL_b)^k + k\int_{AL_b}^\infty z^{k-1} P_1\, \d z + k\int_{AL_b}^\infty z^{k-1} P_2\, \d z
			= : (AL_b)^k + I_1 + I_2,
	\end{align*}
	uniformly for all $(k\,,b\,,A) \in [1\,,\infty) \times (0\,,b_0] \times (0\,,\infty)$.
	We plug in the above estimate for $P_1$ and use a change of variable to deduce that
	\begin{align*}
		&I_1 \lesssim k n^{\frac{1}{\gamma}+\frac{1}{\alpha\gamma}} b^{c_1A^2/2}
			\int_{AL_b}^\infty z^{k-1} \e^{-c_1 z^2/(2b^{2\beta})}\, \d z\\
		&\le k n^{\frac{1}{\gamma}+\frac{1}{\alpha\gamma}} b^{c_1A^2/2} (2b^{2\beta}/c_1)^{\frac{k}{2}} \int_0^\infty y^{\frac{k}{2}-1} \e^{-y}\, \d y
		= (2/c_1)^{\frac{k}{2}} k \Gamma(k/2) n^{\frac{1}{\gamma}+\frac{1}{\alpha\gamma}} b^{\beta k+c_1A^2/2},
	\end{align*}
	uniformly for all $(k\,,b\,,A\,,n) \in [1\,,\infty) \times (0\,,b_0] \times (0\,,\infty) \times \N$.
	Similarly, the above estimate for $P_2$ implies that
	\begin{align*}
		I_2 &\lesssim k \int_{AL_b}^\infty z^{k-1} \e^{-c_2 z^2 n^2 / \log n}\, \d z\\
		& \le k \left( \frac{\log_+ n}{c_2 n^2} \right)^{k/2} 
			\int_{0}^\infty y^{\frac{k}{2}-1} \e^{-y}\, \d y
		= c_2^{-k/2} k \Gamma(k/2) \left( \frac{\sqrt{\log_+ n}}{n} \right)^k,
	\end{align*}
	uniformly for all $(k\,,b\,,A) \in [1\,,\infty) \times (0\,,b_0] \times (0\,,\infty)$.
	There exists $c_0 > 1$ such that $L_b n/\sqrt{\log_+ n} \ge 1$ 
	for $n = n_b = \lfloor c_0 b^{-\beta} \rfloor$. 
	Set $A = \sqrt{2\beta(1+1/\alpha)/(\gamma c_1)}$.
	Then, there exists $C>0$ such that
	\begin{align*}
		&\textstyle\left\|\sup_{t\in [t_0,t_1]} \sup_{h \in (0,b]} |\mathcal{E}_\alpha(t\,,h)| \right\|_k\\
		&\le \left[ (AL_b)^k + C^k \Gamma(k/2) b^{\beta k + 
			c_1 A^2/2 - \frac{\beta}{\gamma}(1+\frac{1}{\alpha})} + C^k \Gamma(k/2) L_b^k \right]^{1/k}
			\lesssim \sqrt{k}\, b^{\beta} |\log b|^{1/2},
	\end{align*}
	uniformly for all $k \in [1\,,\infty)$ and $b \in (0\,,b_0]$.
	This implies that
	\begin{align*}
		\left\|\sup_{t\in [t_0,t_1]} \sup_{h \in (\e^{-m},\e^{-m+1}]} 
			\frac{|\mathcal{E}_\alpha(t\,,h)|}{h^H} \right\|_k 
		&\le \e^{Hm} \left\|\sup_{t\in [t_0,t_1]} \sup_{h \in (0,\e^{-m+1}]}
			|\mathcal{E}_\alpha(t\,,h)| \right\|_k\\
		& \lesssim \sqrt{km} \, \e^{-(\beta-H)m},
	\end{align*}
	uniformly for all $(k\,,m) \in [1\,,\infty) \times \N$. 
	Note that $\beta>H$.
	For every $b \in (0\,,b_0]$, we can find $n \in \N$ such that 
	$b \in (\e^{-n},\e^{-n+1}]$. Thus, we have
	\begin{align*}
		&\left\|\sup_{t\in [t_0,t_1]} \sup_{h \in (0,b]} 
			\frac{|\mathcal{E}_\alpha(t\,,h)|}{h^H} \right\|_k 
			\le \sum_{m=n}^\infty \left\|\sup_{t\in [t_0,t_1]}
			\sup_{h \in (\e^{-m},\e^{-m+1}]} \frac{|\mathcal{E}_\alpha(t\,,h)|}{h^H} \right\|_k\\
		&\lesssim \sqrt{k} \sum_{m=n}^\infty \sqrt{m}\, \e^{-(\beta-H)m}
			\lesssim \sqrt{kn}\, \e^{-(\beta-H)n}\lesssim \sqrt{k}\, b^{\beta-H} |\log b|^{1/2},
	\end{align*}
	uniformly for all $k \in [1\,,\infty)$. This completes the proof.
\end{proof}

Lemma \ref{lem:supE} and the Borel-Cantelli lemma 
readily imply the following.

\begin{proposition}\label{pr:E:as}
	For every $\alpha \in (0\,,1)$ and $t_1>t_0>0$, the following holds a.s.:
	\[\textstyle
		\sup_{t \in [t_0,t_1]} |B(t+h)-B(t) - D_\alpha(t\,,h)| = 
		\mathscr{o}(h^H) \quad \text{as $h \to 0^+$.}
	\]
\end{proposition}
%

\section{Proof of Theorem \ref{th:main}}
From here on out, define
\[
	\tilde{D}_\alpha(t\,,h) = h^{-H} D_\alpha(t\,,h)
	\qquad\forall t,h>0,\ \alpha\in(0\,,1).
\]
We begin by establishing the following analogue of the boundary-crossing estimate \eqref{BC}
for the localized process $D_\alpha$; see \eqref{D_alpha}.

\begin{lemma}\label{lem:BC:D}
	Choose and fix $\alpha \in (0\,,1)$, $t_1 > t_0 > 0$,
	$\theta>0$, and $q>1$. Then,
	\[
		\sup_{t \in [t_0\,,t_1]}
		\P\left\{ | \tilde{D}_\alpha(t\,,h)| \le \theta\ \forall h \in [a\,,b]\right\} =
		(a/b)^{\lambda(\theta) + \mathscr{o}(1)}
	\]
	as $\max\{b\,,a/b\} \to 0^+$ subject to $b^q \le a \le b$.
\end{lemma}

\begin{proof}
	Let $b_0 = t_0^{1/\alpha} \wedge \e^{-1}$.
	Thanks to the triangle inequality, we may write,
	for every $\theta, \eta > 0$, $t \in [t_0\,,t_1]$, and  $0<a<b< b_0$,
	\begin{align}\begin{split}\label{P:BC:D:ub}
		&\P\left\{ | \tilde{D}_\alpha(t\,,h)| \le \theta\ \forall h \in [a\,,b]\right\}\\
		&\hskip1in
			\le \P\left\{ |B(t+h)-B(t)| \le (\theta+\eta) h^H \ 
			\forall h \in [a\,,b]  \right\} + P_0,
	\end{split}\end{align}
	where
	$P_0 = \P\{ \sup_{t \in [t_0,t_1]} \sup_{h\in [a,b]} 
	h^{-H} |\mathcal{E}_\alpha(t\,,h)| > \eta \}$; see \eqref{error}.
	Elementary properties of fBm and \eqref{BC} together imply that
	\begin{equation}\label{inc1}\begin{aligned}
		&\P\left\{ |B(t+h)-B(t)| \le (\theta+\eta) h^H \ 
			\forall h \in [a\,,b]  \right\}\\
		&= \P\left\{ |B(h)| \le (\theta+\eta) h^H \ 
			\forall h \in [a/b\,,1]  \right\}
			=(a/b)^{\lambda(\theta+\eta)+\mathscr{o}(1)},
	\end{aligned}\end{equation}
	as $\max\{b\,,a/b\} \to 0^+$, 
	uniformly for all $t \in [t_0\,,t_1]$.
	It remains to estimate $P_0$.
	
	Thanks to the Taylor expansion of the exponential function, 
	Lemma \ref{lem:supE}, and Stirling's formula, there exists 
	$c>0$ such that
	\begin{align*}
		&\E\left[ \exp\left( \frac{c}{b^{2(\beta-H)}|\log b|} 
			\sup_{t\in [t_0,t_1]} \sup_{h\in(0,b]} \left|\frac{
			\mathcal{E}_\alpha(t\,,h)}{h^H}\right|^2 \right) \right]\\
		& \le \sum_{k=0}^\infty \frac{c^k}{k! (b^{2(\beta-H)} |\log b|)^k}
			\left\| \sup_{t\in[t_0,t_1]} \sup_{h\in(0,b]} \frac{|
			\mathcal{E}_\alpha(t\,,h)|}{h^H} \right\|_{2k}^{2k} < \infty,
	\end{align*}
	uniformly for all $b \in (0\,,b_0]$.
	Thus, the preceding and Chebyshev's inequality together 
	imply that there exists $C>0$ such that
	\begin{align*}
		P_0 \le C \exp\left( - \frac{\eta^2}{C b^{2(\beta-H)}|\log b|} \right),
	\end{align*}
	uniformly for all $\eta > 0$ and $b \in (0\,,b_0]$.
	Combine this with \eqref{inc1}
	and the fact that $\beta>H$ in order
	to deduce that, for every $\theta, \eta >0$ and $0<a<b<b_0$,
	\begin{align*}
		\P\left\{ | \tilde{D}_\alpha(t\,,h)| \le \theta\ \forall h\in [a\,,b] \right\}
			&\le (a/b)^{\lambda(\theta+\eta)+\mathscr{o}(1)} +
			C \exp\left( - \frac{\eta^2}{C b^{2(\beta-H)}|\log b|} \right)\\
		& = (a/b)^{\lambda(\theta+\eta)+\mathscr{o}(1)},
	\end{align*}
	as $\max\{b\,,a/b\} \to 0^+$, subject to
	$b^q \le a \le b$, all valid
	uniformly for all $t \in [t_0\,,t_1]$.
	Similarly, we may deduce the complementary bound: For any $\theta>\eta>0$ and $0<a<b<b_0$,
	\begin{align*}
		&\P\left\{ |\tilde{D}_\alpha(t\,,h)| \le \theta\ \forall h\in [a\,,b] \right\}
			\ge \P\left\{ |B(t+h) - B(t)| \le (\theta-\eta) \right\} - P_0\\
		&\ge (a/b)^{\lambda(\theta-\eta)+\mathscr{o}(1)} - 
			C \exp\left( - \frac{\eta^2}{C b^{2(\beta-H)}|\log b|} \right)
			= (a/b)^{\lambda(\theta-\eta)+\mathscr{o}(1)},
	\end{align*}
	as $\max\{b\,,a/b\} \to 0^+$, subject to
	$b^q \le a \le b$, all valid
	uniformly for all $t \in [t_0\,,t_1]$.
	This proves that for every $\theta>\eta>0$,
	\begin{align*}
		\lambda(\theta+\eta) + \mathscr{o}(1) \le 
		\frac{\log \P\left\{ |\tilde{D}_\alpha(t\,,h)| \le \theta\ 
		\forall h \in [a\,,b] \right\}}{\log(a/b)} \le \lambda(\theta-\eta) + \mathscr{o}(1)
	\end{align*}
	as $\max\{b\,,a/b\} \to 0^+$, subject to
	$b^q \le a \le b$, all valid
	uniformly for all $t \in [t_0\,,t_1]$.
	Since $\lambda$ is continuous, we may let $\eta \to 0^+$ to conclude the proof.
\end{proof}

Armed with the preceding technical results we can now begin to prove
Theorem \ref{th:main}. The proof is naturally divided in two parts: The upper bound;
and the lower bound. We discuss each separately and in turn.

\subsection{Proof of the upper bound}

\begin{proof}
	The proof is based on the second-moment method, and 
	adapts ideas of \cite{KL2} to the present setting.
	
	Choose and fix a compact set $K$ with $K \subset [t_0\,,t_1]$ for some $t_1>t_0>0$.
	Without loss of generality, we may and will assume that $\dimh{K}>0$.
	Choose and fix two numbers $\theta, \nu>0$ such that
	\begin{align}\label{theta}
		\lambda(\theta) < \nu < \dimh{K}.
	\end{align}
	Frostman's lemma \cite{Falconer} ensures that there exists a Borel probability measure $\mu$ on $K$ such that
	\begin{align}\label{sup:mu}
		\sup_{r>0} \sup_{t \in \R} r^{-\nu}  \mu([t-r\,,t+r]) <\infty.
	\end{align}
	Choose and fix a number 
	\begin{align}\label{alpha}
		\alpha\in\left( \lambda(\theta)/\nu~,~1\right).
	\end{align}
	For every $h_2\ge h_1>0$ and $t \in (0\,,\infty)$, define the event
	\begin{align}\label{event:A}
		A(h_1\,,h_2\,,t) = \left\{ \omega \in \Omega : \sup_{h \in [h_1,h_2]} 
		 | \tilde{D}_\alpha(t\,,h)| (\omega)  \le \theta\right\}.
	\end{align}
	By Lemma \ref{lem:BC:D},
	\begin{align}\label{P:BC}
		\P(A(h_1\,,h_2\,,t)) = (h_1/h_2)^{\lambda(\theta)+\mathscr{o}(1)} \ 
		\text{($\max\{ h_2\,,h_1/h_2\} \to 0^+$,\
		$h_2^q \le h_1 \le h_2$)},
	\end{align}
	valid uniformly for all $t \in [t_0\,,t_1]$.
	For every $h_2\ge h_1>0$, consider the random variable
	\[
		X_{h_1,h_2} = \int_K \frac{\1_{A(h_1,h_2,t)}}{\P(A(h_1\,,h_2\,,t))} \mu(\d t).
	\]
	Choose and fix some $q>1$,
	and define $\varepsilon_m = \exp(-q^m)$ for every $m \in \N$.
	We aim to show that $\sum_{m=1}^\infty 
	\P\{X_{\varepsilon_{m+1},\varepsilon_m}=0\} < \infty$ for a suitable choice of $q>1$.
	To this end, we compute
	\[
		\E(X_{h_1,h_2}^2) = \iint_{K \times K} 
		\frac{\P(A(h_1\,,h_2\,,t)\cap A(h_1\,,h_2\,,s))}{
		\P(A(h_1\,,h_2\,,t))\P(A(h_1\,,h_2\,,s))} \mu(\d t) \mu(\d s),
	\]
	and decompose it as
	$\E(X_{h_1,h_2}^2) = I_1(h_1\,,h_2)+I_2(h_1\,,h_2)+I_3(h_1\,,h_2)$,
	where
	\begin{align*}
		&I_1 = I_1(h_1\,,h_2)=\iint_{\substack{s,t \in K:\\|t-s| \le
			2h_1^\alpha}}
			\frac{\P(A(h_1\,,h_2\,,t)\cap A(h_1\,,h_2\,,s))}{
			\P(A(h_1\,,h_2\,,t))\P(A(h_1\,,h_2\,,s))} \mu(\d t) \mu(\d s),\\
		&I_2 = I_2(h_1\,,h_2)=\iint_{\substack{s,t \in K:\\ 
			2h_1^\alpha< |t-s| \le 2h_2^\alpha}} \frac{\P(A(h_1\,,h_2\,,t)\cap 
			A(h_1\,,h_2\,,s))}{\P(A(h_1\,,h_2\,,t))
			\P(A(h_1\,,h_2\,,s))} \mu(\d t) \mu(\d s).
	\end{align*}
	To estimate $I_1$, we use the trivial bound $\P(A_1\cap A_2) \le \P(A_1)$,
	valid for all events $A_1$ and $A_2$, together with \eqref{sup:mu} and \eqref{P:BC}.
	This yields
	\begin{align*}
		I_1 &\le \iint_{\substack{s,t \in K:\\|t-s| \le 2h_1^\alpha}}
			\frac{\mu(\d t) \mu(\d s)}{\P(A(h_1\,,h_2\,,t))} = 
			\iint_{\substack{s,t \in K:\\|t-s| \le 2h_1^\alpha}}
			\frac{\mu(\d t) \mu(\d s)}{(h_1/h_2)^{\lambda(\theta)+\mathscr{o}(1)}}\\
		&\lesssim \frac{h_1^{\alpha\nu}}{(h_1/h_2)^{\lambda(\theta)+\mathscr{o}(1)}} 
			= h_1^{\alpha\nu-\lambda(\theta)+\mathscr{o}(1)} 
			h_2^{\lambda(\theta)+\mathscr{o}(1)} 
	\end{align*}
	as $\max\{h_2\,,h_1/h_2\} \to 0^+$ via $h_2^q \le h_1 \le h_2$.
	Next, we observe that if $|t-s| > 2h_2^\alpha$, then
	$A(h_1\,,h_2\,,t)$ and $A(h_1\,,h_2\,,s)$ are independent events;
	see \eqref{D_alpha}. Therefore,
	$I_3 \le 1.$
	We now estimate $I_2$ as follows. 
	If $2h_1^\alpha< |t-s| \le 2 h_2^\alpha$, then
	\begin{align*}
		&\P\left\{ \sup_{h \in [h_1,h_2]} | \tilde{D}_\alpha(s\,,h)|
			\le \theta~~,\sup_{h \in [h_1,h_2]} |\tilde{D}_\alpha(s\,,h)|
			\le \theta \right\}\\
		& \le \P\left\{ \sup_{h \in [h_1,(|t-s|/2)^{1/\alpha}]} 
			|\tilde{D}_\alpha(s\,,h)| \le \theta~~,
			\sup_{h \in [h_1,(|t-s|/2)^{1/\alpha}]} 
			|\tilde{D}_\alpha(s\,,h)| \le \theta \right\}\\
		& = \P\left\{ \sup_{h \in [h_1,(|t-s|/2)^{1/\alpha}]}
			|\tilde{D}_\alpha(s\,,h)| \le \theta\right\}
			\P\left\{ \sup_{h \in [h_1,(|t-s|/2)^{1/\alpha}]}
			|\tilde{D}_\alpha(s\,,h)| \le \theta\right\},
	\end{align*}
	where the last line follows from independence of the two events.
	By the Gaussian correlation inequality of Royen \cite{Royen},
	\begin{align*}
		\P\left\{ \sup_{h \in [h_1,(|t-s|/2)^{1/\alpha}]}
			|\tilde{D}_\alpha(u\,,h)| \le \theta\right\}
		& \le \frac{\P\left\{ \sup_{h \in [h_1,h_2]} 
			|\tilde{D}_\alpha(u\,,h)| \le \theta \right\}}{
			\P\left\{ \sup_{h \in [(|t-s|/2)^{1/\alpha},h_2]} 
			|\tilde{D}_\alpha(u\,,h)| \le \theta \right\}},
	\end{align*}
	for $u\in\{t\,,s\}$.
	Because of \eqref{alpha}, it is possible to choose and fix $\delta>0$ such that
	\begin{align}\label{delta}
		(\lambda(\theta)+\delta)/\alpha< \nu.
	\end{align}
	Thanks to \eqref{P:BC} and \eqref{sup:mu}, the above implies that
	\begin{align*}
		I_2
		&\le \iint_{\substack{s,t \in K:\\ 2h_1^\alpha < |t-s| \le 2h_2^\alpha}} 
			\frac{\mu(\d t) \mu(\d s)}{\P(A((|t-s|/2)^{1/\alpha}\,,h_2\,,t))
			\P(A((|t-s|/2)^{1/\alpha}\,,h_2\,,s))}\\
		&\le \iint_{\substack{s,t \in K:\\ 2h_1^\alpha < |t-s| \le 2h_2^\alpha}} 
			\left( \frac{h_2}{(|t-s|/2)^{1/\alpha}} \right)^{
			2\lambda(\theta)+2\delta} \mu(\d t) \mu(\d s)\\
		&\lesssim h_2^{2\lambda(\theta)+2\delta} 
			h_1^{-2\lambda(\theta)-2\delta} h_2^{\alpha \nu} \qquad 
			\text{($\max\{h_2\,,h_1/h_2\} \to 0^+$, $h_2^q \le h_1 \le h_2$).}
	\end{align*}
	We now choose $q>1$ sufficiently close to 1 so that
	$\alpha\nu - (q-1)(2\lambda(\theta)+2\delta) > 0$,
	$\varepsilon_{m+1} = \varepsilon_m^q$, and 
	$\max\{\varepsilon_m\,,\varepsilon_{m+1}/\varepsilon_m\}\to 0^+$. Thanks to the above estimates, 
	\begin{align*}
		\E(X_{\varepsilon_{m+1},\varepsilon_m}^2) \le 1 + 
		\e^{-q^m(q\alpha\nu + (q+1)\lambda(\theta)+\mathscr{o}(1))} + 
		\e^{-q^m(\alpha\nu - (q-1)(2\lambda(\theta)+2\delta))},
	\end{align*}
	as $m\to\infty$.
	The Paley-Zygmund inequality, the fact that 
	$\E(X_{\varepsilon_{m+1},\varepsilon_m})=1$, and the preceding together imply that
	\begin{align*}
		\P\{X_{\varepsilon_{m+1},\varepsilon_m} &= 0\} 
			= 1 - \P\{X_{\varepsilon_{m+1},\varepsilon_m} > 0\} \le 1 - 
			\frac{[\E(X_{\varepsilon_{m+1},\varepsilon_m})]^2}{%
			\E(X_{\varepsilon_{m+1},\varepsilon_m}^2)}\\
		&\le \e^{-q^m(q\alpha\nu + (q+1)\lambda(\theta)+\mathscr{o}(1))}
			+ \e^{-q^m(\alpha\nu - (q-1)(2\lambda(\theta)+2\delta))}
			\quad \text{as $m \to \infty$.}
	\end{align*}
	Hence, $\sum_{m=1}^\infty \P\{X_{\varepsilon_{m+1},\varepsilon_m}=0\} < \infty$ 
	and the Borel-Cantelli lemma yields
	\[
		\P\Bigg( \bigcup_{n \in \N} \bigcap_{m \ge n} \left\{ 
		\exists\, t \in K, \sup_{h \in [\varepsilon_{m+1},\varepsilon_m]} 
		|\tilde{D}_\alpha(t\,,h)| \le \theta \right\}\Bigg) = 1.
	\]
	Then, by the continuity of $D_\alpha$ 
	-- see Lemma \ref{lem:E-E} and \eqref{error} --
	and a compactness argument as in \cite{KL2}, we deduce that
	\[
		\bigcup_{n \in \N} \bigcap_{N>n} \left\{ t \in K : 
		\sup_{h\in[\varepsilon_N,\varepsilon_n]} |\tilde{D}_\alpha(t\,,h)|
		 \le \theta\right\} 
		\ne \varnothing\qquad\text{a.s.}
	\]
	This, together with Proposition \ref{pr:E:as}, implies that, a.s.,
	\[
		\adjustlimits\inf_{t \in K} \limsup_{h\to0^+}\frac{|B(t+h)-B(t)|}{h^H} 
		= \adjustlimits\inf_{t \in K} \limsup_{h\to0^+} |
		\tilde{D}_\alpha(t\,,h)| \le \theta.
	\]
	The preceding is valid for every fixed $\theta,\nu>0$ that satisfy \eqref{theta}.
	We may let $\nu \uparrow \dimh K$ along rationals see that 
	$\inf_{t \in K} \limsup_{h\to0^+} h^{-H}|B(t+h)-B(t)| \le \theta$ for 
	every $\theta>0$ that satisfies $\lambda(\theta)< \dimh K$, off a $\P$-null set,
	Since $\lambda$ is strictly decreasing, we may take infimum over all such 
	$\theta$ to conclude the upper bound for the $\limsup$ in \eqref{main}.
\end{proof}

\subsection{Proof of the lower bound}

We follow the proof of the lower bound in \cite{KL2} but make suitable
non-trivial changes that we report next.
Choose and fix a compact set $K$ with $K \subset[t_0\,,t_1]$ for some fixed numbers $t_1>t_0>0$. 
Choose and fix numbers $\nu,p,\theta,\eta>0$ such that
\begin{align}\label{theta:eta}
	\nu > \dimmm K, \quad p > 1, \quad \text{and} \quad
	\lambda(\theta+\eta) > p\nu.
\end{align}
Recall that
$\dimmm K = \liminf_{h\to0^+}\log \mathscr{P}_h(K)/|\log h|,$
where $\mathscr{P}_h(K)$ denotes the cardinality of the smallest $h$-packing of $K$, and a finite set $L$ is called an $h$-packing of $K$ if every $x \in K$ is within $h$ of some $y \in L$.
Since $\nu>\dimmm K$, there is a positive sequence $\delta_1>\delta_2>\dots$ tending to 0 such that
\begin{align}\label{packing}
	\mathscr{P}_{\delta_m}(K) \le \delta_m^{-\nu} \qquad \forall m \in \N.
\end{align}
Due to \eqref{theta:eta}, we may choose and fix $\alpha \in (0\,,1)$, sufficiently close to $1$, such that
\begin{align}\label{alpha2}
	p/2 - H + \alpha(H-1/2) > 0
\end{align}
as well as a large-enough $\mu>1$ such that
\begin{align}\label{mu}
	\mu(\lambda(\theta+\eta)-p\nu)
		>\lambda(\theta+\eta) + \alpha.
\end{align}
We extract a subsequence $\{\tilde{\delta}_n\}_{n \in \N}$ of $\{\delta_m\}_{m\in \N}$ such that
$h_n := \tilde{\delta}_n^{1/p}$ $\forall n \in \N$
satisfies
\begin{align}\label{eps}
	h_{n} \le \exp(-\mu^n) \quad \text{and} \quad
	h_{n+1} \le h_n^\mu \qquad \forall n \in \N.
\end{align}
Thanks to \eqref{packing}, for every $n \in \N$, there exists a finite subset $F_n \subset K$ such that the following properties hold:
\begin{enumerate}[\bf Property A.]
	\item For every $x \in K$, there exists $y \in F_n$ such that 
		$|x-y| \le h_{n+1}^p$;	\label{pack}
	\item Whenever $y\,,z \in F_n$ are distinct, 
		$|y-z| \ge 2h_{n+1}^p$; and
	\item $|F_n| \asymp h_{n+1}^{-p\nu}$, 
		where $|\cdot|$ denotes cardinality.
	\label{|F|}
\end{enumerate}
Next, we consider a uniform partition of the interval $[t_0\,,t_1]$ as follows:
\[
	[t_0\,,t_1] = [t_{n,1}\,, t_{n,2}] \cup [t_{n,2}\,,t_{n,3}] \cup \cdots \cup [t_{n,N}\,,t_{n,N+1}],
\]
where
\begin{align}\label{t_ni}
	t_{n,i+1}-t_{n,i} \ge 2 h_n^{\alpha} \quad \text{and} \quad
	N=N(n) \asymp h_n^{-\alpha}.
\end{align}
For every $1\le i \le N$,  define
$K_i = K_{n,i} = K \cap [t_{n,i}\,,t_{n,i+1}]$.
By Property \ref{|F|}, 
\begin{align}\label{FnK}
	|F_n \cap K_i| \lesssim h_{n+1}^{-p\nu} \qquad (1 \le i \le N).
\end{align}
We may write $K = I \cup J$, where
$I = K_1 \cup K_3 \cup \dots$, and
$J = K_2 \cup K_4 \cup \dots.$
Then, by Boole's inequality,
\begin{align*}
	&\P\left\{ \adjustlimits\inf_{t \in K} \sup_{h\in[h_{n+1},h_n]} 
		|\tilde{D}_\alpha(t\,,h)| \le \theta \right\}\\
	&\le \P\left\{ \adjustlimits\inf_{t \in I} \sup_{h\in[h_{n+1},h_n]} 
		|\tilde{D}_\alpha(t\,,h)| \le \theta \right\} + 
		\P\left\{ \adjustlimits\inf_{t \in J} \sup_{h\in[h_{n+1},h_n]} 
		|\tilde{D}_\alpha(t\,,h)| \le \theta \right\}.
\end{align*}
Because $\{D_\alpha(t\,,h)\}_{t \in K_i, h\in[h_{n+1},h_n]}$
and $\{D_\alpha(t\,,h)\}_{t \in K_j, h\in[h_{n+1},h_n]}$
are independent whenever $|i-j| \ge 2$.,
\begin{align*}
	&\P\left\{ \adjustlimits\inf_{t \in I} \sup_{h\in[h_{n+1},h_n]}
		|\tilde{D}_\alpha(t\,,h)| \le \theta \right\}
		= 1 - \P\left\{ \adjustlimits\inf_{t \in I} \sup_{h\in[h_{n+1},h_n]} 
		|\tilde{D}_\alpha(t\,,h)| > \theta \right\}\\
	& = 1 - \prod_{i : K_i \subset I} \P\left\{ \adjustlimits\inf_{t \in K_i} 
		\sup_{h\in[h_{n+1},h_n]} |\tilde{D}_\alpha(t\,,h)| > \theta \right\}\\
	& = 1 - \prod_{i : K_i \subset I} \left[ 1 - \P\left\{ \adjustlimits\inf_{t \in K_i} 
		\sup_{h\in[h_{n+1},h_n]} |\tilde{D}_\alpha(t\,,h)|
		\le \theta \right\} \right].
\end{align*}
Fix an index $i \in \N$ such that $K_i \subset I$.
We may use Property \ref{pack} and interpolation to write
\begin{align*}
	\P\left\{ \adjustlimits\inf_{t \in K_i} \sup_{h\in[h_{n+1},h_n]} 
		|\tilde{D}_\alpha(t\,,h)| \le \theta \right\} 
	\le \P\left\{ \adjustlimits\min_{s \in F_n \cap K_i} 
	\sup_{h\in[h_{n+1},h_n]} |\tilde{D}_\alpha(s\,,h)| \le \theta + \eta \right\}\\
	+ \P \left\{ \sup_{\substack{s,t\in [t_0,t_1]:\\ |s-t|\le 
	h_{n+1}^p}}\sup_{h\in [h_{n+1},h_n]} 
	|\tilde{D}_\alpha(t\,,h)-\tilde{D}_\alpha(s\,,h)| > \eta\right\}.
\end{align*}
We estimate the last two probabilities in turn.
First, by \eqref{FnK} and Lemma \ref{lem:BC:D}, 
\begin{align*}
	&\P\left\{ \adjustlimits\min_{s \in F_n \cap K_i} 
		\sup_{h\in[h_{n+1},h_n]} |\tilde{D}_\alpha(s\,,h)|
		\le \theta + \eta \right\}\\
	&\hskip1in\lesssim h_{n+1}^{-p\nu} (h_{n+1}/h_n)^{\lambda(\theta+\eta)+
		\mathscr{o}(1)} = h_{n+1}^{\lambda(\theta+\eta)-p\nu+
		\mathscr{o}(1)} h_n^{-\lambda(\theta+\eta)+\mathscr{o}(1)},
\end{align*}
as $n \to \infty$. 
Next, define
$d((t\,,h)\,,(t',h')) := \|D_\alpha(t\,,h)-D_\alpha(t',h')\|_2.$
By Lemmas \ref{lem:D-D:t} and \ref{lem:D-D:eps},
\begin{align}\label{D-D}
	d((t\,,h)\,,(t',h')) \lesssim |t-t'|^{H \wedge \frac12} + |h-h'|^{\alpha(H\wedge \frac12)},
\end{align}
uniformly for all $t, t' \in [t_0\,,t_1]$ and $h, h' \in (0\,,t_0^{1/\alpha}]$.
Choose and fix $\gamma \in (0\,,1)$, sufficiently close to $1$, such that
\begin{align}\label{gamma}
	\delta:=\begin{cases}
	(\gamma p - 1)H > 0, & \text{if $H \in (0\,,\frac12]$,}\\
	\gamma p/2 - H + \gamma \alpha(H-1/2) > 0, & \text{if $H \in (\frac12\,,1)$,}
	\end{cases}
\end{align}
which is possible for both cases due to $p>1$ and \eqref{alpha2}, respectively.
Then, \eqref{D-D} together with a standard metric entropy argument yields a constant $c>0$ such that
\begin{align}\label{Eexp:D}
	\E\left[ \exp\left( c \sup_{(t,h)\ne(t',h') \in [t_0,t_1] \times (0,t_0^{1/\alpha}]} \left|\frac{D_\alpha(t\,,h)-D_\alpha(t',h')}{d^\gamma((t\,,h)\,,(t',h'))}\right|^2 \right) \right]<\infty.
\end{align}
Moreover, it follows from Lemma \ref{lem:D-D:t} that
\begin{align}\label{d:ub}
	d((t\,,h)\,,(s\,,h)) \lesssim \begin{cases}
	|t-s|^H & \text{if $H \in (0\,,\frac12]$,}\\
	h^{\alpha(H-1/2)} |t-s|^{1/2} & \text{if $H \in (\frac12\,,1)$,}
	\end{cases}
\end{align}
uniformly for all $t, s \in [t_0\,,t_1]$ with $|t-s|\le h_{n+1}^p$, for 
all $h \in [h_{n+1}\,,h_n]$, and for all $n \in \N$.
Hence, \eqref{gamma}-\eqref{d:ub} and Chebyshev's inequality 
together imply that there exists a constant $c_0>0$ such that the following holds: If $H \in (0\,,\frac12]$, then
\begin{align*}
	&\mathscr{P}_n:= \P \left\{ \sup_{\substack{s,t\in [t_0,t_1]:\\ |s-t|\le h_{n+1}^p}}
		\sup_{h\in [h_{n+1},h_n]} |\tilde{D}_\alpha(t\,,h)-
		\tilde{D}_\alpha(s\,,h)| > \eta\right\}\\
	& \le \P \left\{ \sup_{\substack{s,t\in [t_0,t_1]:\\ |s-t|\le 
		h_{n+1}^p}}\sup_{h\in [h_{n+1},h_n]} \frac{
		|D_\alpha(t\,,h)-D_\alpha(s\,,h)|}{d^\gamma((t\,,h)\,,(s\,,h))} 
		> \frac{\eta h_{n+1}^{H}}{c_0h_{n+1}^{\gamma p H}}\right\}
		\lesssim \e^{-\eta^2/c_0^2 h_{n+1}^{2\delta}},
\end{align*}
uniformly for all $n \in \N$; and if $H \in (\frac12\,,1)$, then
\begin{align*}
	&\mathscr{P}_n
		\le \P \left\{ \sup_{\substack{s,t\in [t_0,t_1]:\\ |s-t|\le h_{n+1}^p}}
		\sup_{h\in [h_{n+1},h_n]} \frac{|D_\alpha(t\,,h)-D_\alpha(s\,,h)|}{
		h^{\gamma\alpha(H-1/2)}} > \eta h_{n+1}^{H-\gamma\alpha(H-1/2)}\right\}\\
	& \le \P \left\{ \sup_{\substack{s,t\in [t_0,t_1]:\\ |s-t|\le h_{n+1}^p}}
		\sup_{h\in [h_{n+1},h_n]} \frac{|D_\alpha(t\,,h)-D_\alpha(s\,,h)|}{
		d^\gamma((t\,,h)\,,(s\,,h))} > \frac{\eta h_{n+1}^{H-\gamma
		\alpha(H-1/2)}}{c_0h_{n+1}^{\gamma p/2}}\right\}
		\lesssim \e^{- \eta^2/c_0^2 h_{n+1}^{2\delta}},
\end{align*}
uniformly for all $n \in \N$. It follows by combining the above estimates that
\begin{align*}
	&\P\left\{ \adjustlimits\inf_{t \in K_i} \sup_{h\in[h_{n+1},h_n]}
		|\tilde{D}_\alpha(t\,,h)| \le \theta \right\} \\
	&\lesssim h_{n+1}^{\lambda(\theta+\eta)-p\nu+\mathscr{o}(1)} 
		h_n^{-\lambda(\theta+\eta)+\mathscr{o}(1)} + 
		\exp\left( - \frac{\eta^2}{c_0^2 h_{n+1}^{2\delta}} \right) \quad \text{as $n \to \infty$.}
\end{align*}
The preceding estimate is valid for every index $i\in\N$ that satisfies either
$K_i \subset I$ or $K_i \subset J$.
Therefore, we may apply the above estimates, as well as \eqref{t_ni}, 
in order to see that there exists a constant $C>0$ such that, as $n\to\infty$,
\begin{align*}
	&\P\left\{\adjustlimits
		\inf_{t \in K} \sup_{h \in [h_{n+1},h_n]}
		|\tilde{D}_\alpha(t\,,h)| \le \theta \right\}\\
	& \le 2 \left[ 1 - \left( 1 - C 
		h_{n+1}^{\lambda(\theta+\eta)-p\nu
		+ \mathscr{o}(1)} h_n^{-\lambda(\theta+\eta) + \mathscr{o}(1)}
	+ C\exp\left\{- \frac{\eta^2 }{Ch_{n+1}^{2\delta}}\right\}
		\right)^{C h_n^{-\alpha}} \right].
\end{align*}
It is possible to show that, thanks \eqref{mu} and \eqref{eps}, 
$h_{n+1}^{\lambda(\theta+\eta)-p\nu 
+ \mathscr{o}(1)} h_n^{-\lambda(\theta+\eta) + \mathscr{o}(1)}
=\mathscr{o}(1)$ as $n\to\infty$.
Therefore, the elementary inequality $\exp(-z) \le 1-(z/2)$,
valid for all $z \in [0\,,1]$, together with \eqref{mu} and \eqref{eps}, yields
\begin{align*}
	&\P\left\{ \adjustlimits
		\inf_{t \in K} \sup_{h \in [h_{n+1},h_n]}
		|\tilde{D}_\alpha(t\,,h)| \le \theta \right\}\\
	&\le2\left[ 1 - \exp\left\{ -2C^2 h_n^{-\alpha}
		\left(h_{n}^{\mu(\lambda(\theta+\eta)-
		p\nu)-\lambda(\theta+\eta)
		+\mathscr{o}(1)} + \exp\left[-\frac{\eta^2}{Ch_{n+1}^{2\delta}}\right]
		\right)\right\} \right]\\
	&\le 2\left[ 1 - \exp\left\{ -2C^2 \left(\e^{ - \mu^n 
		[ \mu ( \lambda(\theta+\eta) - p\nu ) -\alpha - 
		\lambda(\theta+\eta) + \mathscr{o}(1) ]} + 
		h_n^{-\alpha} 
		\exp\left[-\frac{\eta^2}{Ch_{n+1}^{2\delta}}\right]
		\right)\right\} \right]\\
	&=\mathscr{o}(1) \quad \text{as $n\to\infty$.}
\end{align*}
Let $A_n$ denote the complement of the event in the first line of the last display.
We have shown that $\liminf_{n\to\infty}\P(A_n^c)=0$.
Since $\P\{A_n\text{\ \rm i.o.}\} \ge 1-\liminf_{n\to\infty}\P(A_n^c),$
it follows that, a.s., infinitely many of the events $A_n$ must occur. 
Thus, we can deduce from Proposition \ref{pr:E:as} that, a.s.,
\begin{align}\label{E:inflim:LB}
	\adjustlimits\inf_{t\in K}\limsup_{h \to 0^+} h^{-H} |B(t+h)-B(t)| 
	= \adjustlimits\inf_{t\in K}\limsup_{h \to 0^+} 
	|\tilde{D}_\alpha(t\,,h)| \ge \theta.
\end{align}
Recall that $\theta,\eta,p,\nu$ satisfy \eqref{theta:eta}.
Because $\lambda$ is strictly decreasing and continuous, 
we can let $\nu \downarrow \dimmm K$, $p \downarrow 1$, and $\eta \downarrow 0$ in order
to see that
\eqref{E:inflim:LB} holds for all $\theta>0$ such that $\lambda(\theta)> \dimmm K$.
Take supremum over all such $\theta$ and use monotonicity of $\lambda$ in order to find that
$\inf_{t\in K}\limsup_{h \to 0^+} 
h^{-H} |B(t+h)-B(t)| \ge \lambda^{-1}(\dimmm K)$ a.s.,
where we define $\lambda^{-1}(0)=\infty$.
This completes the proof of the lower bound for the $\limsup$ in \eqref{main} and hence the proof of Theorem \ref{th:main}.
\qed

\bibliographystyle{amsplain}
\bibliography{fbm}

\end{document}